\documentclass[a4paper,11pt]{amsart}
\usepackage[utf8]{inputenc}
\usepackage{graphicx}
\usepackage{amsmath}
\usepackage{amssymb}
\usepackage{amsthm}
\usepackage{color}
\usepackage[hmarginratio=1:1]{geometry}
\usepackage{tikz}
\usepackage{pst-all}
\usetikzlibrary{trees,snakes}

\newcommand{\Ell}{\mathrm{Ell}}
\newcommand{\Hyp}{\mathrm{Hyp}}
\newcommand{\Inv}{\mathrm{Inv}}

\DeclareMathOperator\longleftharpoondown{\pspicture(0,0)(0,0)\psset{linewidth=0.pt}\psline[](0.2,0.15)(-0.2,0.15)\psline[](-0.2,0.15)(-0.125,0.075)\endpspicture}

\newcommand{\PGL}{\mathrm{PGL}}

\newcommand{\conv}{\mathrm{conv}}

\newcommand{\arrow}{\rightarrow}

\DeclareMathOperator{\id}{id}
\DeclareMathOperator{\St}{St}
\DeclareMathOperator{\Aut}{Aut}

\DeclareMathOperator{\Fix}{Fix}
\DeclareMathOperator{\Sym}{Sym}

\newtheorem{theorem}{Theorem}[section]
\newtheorem{lemma}[theorem]{Lemma}
\newtheorem{proposition}[theorem]{Proposition}
\newtheorem{corollary}[theorem]{Corollary}
\theoremstyle{definition}
\newtheorem{definition}[theorem]{Definition}
\theoremstyle{remark}
\newtheorem{remark}[theorem]{Remark}
\newtheorem{example}[theorem]{Example}

\title{Automorphism groups of trees:  generalities and prescribed local actions}
\author{Alejandra Garrido, Yair Glasner and Stephan Tornier}

\begin{document}

\begin{abstract}
This article is an expanded version of the talks given by the authors at the Arbeitsgemeinschaft ``Totally Disconnected Groups'', held at Oberwolfach in October 2014. We recall the basic theory of automorphisms of trees and Tits' simplicity theorem, 
and present two constructions of tree groups via local actions with their basic properties:
the universal group associated to a finite permutation group by M.\ Burger and S.\ Mozes,
and the $k$-closures of a given group by C.\ Banks, M.\ Elder and G.\ Willis.
\end{abstract}

\maketitle

\vspace{-1.1cm}
\section{Introduction}

In the study of totally disconnected locally compact (t.d.l.c.) Hausdorff groups,
groups of automorphisms of locally finite trees appear naturally and form a significant class of examples. 
They are also the most basic case of groups acting on buildings or CAT(0) cube complexes.
Moreover, they play an important role in the structure theory of compactly generated t.d.l.c.\ groups through Schreier graph constructions
(see \cite[\S 11]{Mon} and M.\ Burger's chapter).

In Section 2, we present general results and notions about automorphisms of infinite, locally finite trees. Section \ref{sec:Tits} deals with independence properties 
(the idea that restrictions of the action to subtrees are in some sense independent of each other)
and a general form of Tits' simplicity theorem, Theorem \ref{Tits-k}.
This theorem extracts abstractly simple subgroups from groups acting on the tree in a sufficiently dense and independent way.

We then present two constructions of closed subgroups of tree automorphisms with prescribed local actions
and examine some of the local-global phenomena that they exhibit.
First, we present in Section \ref{sec:universal} the universal group construction by M. Burger and S. Mozes (see \cite{BM00a}),
which to every permutation group $F\le S_{d}$ associates a subgroup $\mathrm{U}(F)\le\Aut(T_{d})$
which locally acts like $F$ (here, $T_{d}$ denotes the $d$-regular tree).
Properties of the groups $\mathrm{U}(F)$ are often determined by those of the finite group $F$.
Universal groups are fundamental in the study of lattices in the product of the automorphism groups of two trees,
analogous to the study of lattices in semisimple Lie groups, 
and are key to the proof of the normal subgroup theorem (see \cite{BM00b} and L.\ Bartholdi's chapter). Second, we describe in Section \ref{sec:k-closure} a variation of this construction by C. Banks, M. Elder and G. Willis (see \cite{BEW11}): Given $G\le\Aut(T)$ for an arbitrary tree $T$ it produces a sequence of closed subgroups $G^{(k)}\le\Aut(T)$ which act like $G$ on balls of radius $k$. This sequence converges to the topological closure $\overline{G}$ of $G$ from above, in a sense to be made precise. We will use Tits' simplicity theorem to construct infinitely many,
pairwise distinct, non-discrete, compactly generated, abstractly simple, t.d.l.c.\ groups. Finding simple t.d.l.c.\ groups (and in particular, compactly generated ones) has become relevant to the structure theory of t.d.l.c.\ groups 
thanks to results of Caprace--Monod \cite{CM11}, which state that such groups can be decomposed into simple pieces.

\section{Generalities on trees}\label{generalities}
\subsection{Trees} Let $T$ be a simplicial tree with vertex set $V(T)$ and edge set $E(T)\subset V(T)\times V(T)$.
We will always assume that $V(T)$ is countable and that $T$ is not isomorphic to a bi-infinite line.
Each edge $e\in E(T)$ is determined by its origin $o(e)\in V(T)$ and its terminus $t(e)\in V(T)$,
so the edge $\bar{e}=(t(e),o(e))$ is the \emph{inverse} of $e$. 
The pair $\{e, \bar{e}\}$ is a \emph{geometric edge}.
Given any edge $e=(x,y)\in E(T)$, the subgraph $T\setminus \{e,\bar{e}\}$ has two connected components $T_x, T_y$ which we call \emph{half-trees}.
We think of $T_x$ as a subtree rooted at $x$.

A \emph{path} is given by a sequence of adjacent vertices, and it is a \emph{reduced path} (or \emph{geodesic}) if there is no backtracking in the sequence.
Setting adjacent vertices to have distance 1 between them yields the standard metric on $T$, denoted by $d$. 
Given a vertex $x\in V(T)$, we set $E(x):=\{e\in E(T)\mid o(e)=x\}$ to be the set of edges starting at $x$ and write $B(x,n)$ for the ball of radius $n$ centred at $x$ as well as $S(x,n)=\{y\in V(T)\mid d(x,y)=n\}$ for the sphere of radius $n$ around $x$.

An isometric embedding of $\mathbb{R}_{\ge 0}$ into $T$ is called a {\it{ray}}. We will say that two rays $\alpha, \beta: \mathbb{R}_{\ge 0} \arrow T$ are equivalent, 
denoting it by $\alpha \sim \beta$, if there exists some $R\in\mathbb{R}$ such that $\alpha(r+s) = \beta(s), \forall r \ge R$.
We call the collection of equivalence classes of such rays the {\it boundary} of the tree and denote it by
$$\partial T = \{\alpha : \mathbb{R}_{\ge 0} \arrow T\}/\sim.$$
The rays $\alpha, \beta$ are \emph{at bounded distance from each other} if the function $f(t) = d(\alpha(t),\beta(t))$ is bounded. This \emph{a priori} weaker equivalence relation is actually the same as the previous one:
it is impossible for two rays to be at bounded distance from each other without actually coinciding eventually. 

We introduce a topology on $\partial T$ (resp. on $T \cup \partial T$) by taking $\{\partial Y \ | \ Y {\text{ is half a tree}}\}$ (resp. $\{Y \cup \partial Y \ | \ Y {\text{ is half a tree}}\})$ as a basis of open sets. It is easy to see that this makes $T$ into an discrete, open and dense subset of $T \cup \partial T$. When $T$ is locally finite then both $\partial T$ and $T \cup \partial T$ are compact with this topology. 

Let $\Aut(T)$ be the automorphism group of $T$. 
Suppose that $G\leq \Aut(T)$ and that $Y$ is a subgraph of $T$. 
The \emph{stabilizer} $\St_G(Y)$ of $Y$ consists of elements $g\in G$ such that $g Y=Y$, while
the \emph{fixator} $\Fix_G(Y)$ of $Y$ consists of $g\in G$ such that $g y=y$ for every vertex $y$ of $Y$. 
We say that $G$ acts \emph{edge-transitively} (respectively, \emph{vertex-transitively}) if $G$ acts transitively on geometric edges (respectively, vertices).

The topology of pointwise convergence on vertices gives $\Aut(T)$ the structure of a Polish (i.e. metrizable, separable and complete) topological group.
The basic open sets for this topology are of the form
$$\mathcal{U}(g,\mathcal{F}):=\{ h\in \Aut(T) \mid hx=gx \text{ for all } x\in \mathcal{F}\},$$
where $g \in \Aut(T)$ and $\mathcal{F}$ is a finite subset of $V(T)$.
In particular, $\Fix(\mathcal{F})$ is open for any finite $\mathcal{F}$.
We will sometimes assume that the tree is locally finite (i.e.\ every vertex has only finitely many neighbors).
In this case the balls $(B(x,n))_n$ around any vertex $x$ are finite.
Since these balls are preserved by the stabilizer $\St(x)$, this open subgroup is the inverse limit of its restriction to balls of finite radius and is therefore compact.
Thus when $T$ is locally finite $\Aut(T)$ is a t.d.l.c.\ group.

\subsection{Classification of automorphisms}
We will say that two directed edges are {\it{co-oriented}} if their orientations agree along the unique geodesic connecting them (see Figure \ref{fig:co_oriented}). 
For an automorphism $\phi \in \Aut(T)$, define $\ell(\phi) = \min\{d(x,\phi x) \ | \ x \in T \}$ and $X(\phi) = \{x \in T \ | \ d(x,\phi(x)) = \ell(\phi)\}$. 
The minimum above is taken over all points $x$ in the geometric realization of the tree.
Thus, for example, if $\phi$ inverts an edge $e$ then the set $X(\phi)$ will be the midpoint of the corresponding geometric edge.
Due to the simplicial nature of the action, this minimal set is actually realized and we are allowed to talk about the minimum rather than the infimum of the translation length. 
\begin{definition} \label{def:classes}
An automorphism $\phi$ is called {\it{hyperbolic}} if $\ell(\phi) > 0$, an {\it{inversion}} if it inverts an edge and {\it{elliptic}} if it fixes some vertex.
It is clear that $\Aut(T)$ is a disjoint union of these three classes and that all three classes, which we denote by $\Hyp, \Inv, \Ell \subset \Aut(T)$, are clopen.
\end{definition}

Let $\phi \in \Aut(T)$ and assume that there exists an edge $e$ that is co-oriented with its image $\phi e$.
Since being co-oriented is an equivalence relation on directed edges, it follows that $\{\phi^{n} e \ | \ n \in \mathbb{Z} \}$ are all co-oriented along a bi-infinite geodesic $X$ and $\phi$ restricts to a translation of length $\ell = d(e, \phi e)+1$ along this axis.
Once we have such an invariant axis, then for every vertex $x \in T$ we have $d(x,\phi x) = \ell + 2 d(x,X) \ge \ell$ (see Figure \ref{fig:co_oriented}).
In particular, as our notation already suggests, $X= X(\phi), \ell = \ell(\phi)$. 
\begin{figure}[ht] 
\centering \def\svgwidth{300pt}
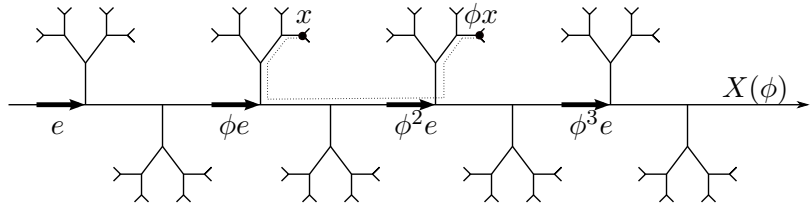 
\caption{A hyperbolic element.}
\label{fig:co_oriented}
\end{figure}
Thus, whenever there exists an edge that is co-oriented with its image, $\phi$ is hyperbolic.
Conversely, if $\ell(\phi) > 0$, let $x \in X(\phi)$, and let $x=x_0, x_1,\ldots, x_{\ell(\phi)}$ be the geodesic connecting $x$ to $\phi x$. 
The directed edge $e = (x_0,x_1)$ must be co-oriented with its image since otherwise we would have $d(x_1,\phi x_1) < d(x, \phi x)$,
contradicting our choice of $x$.
A similar picture shows that the general formula $d(x,\phi x) = \ell(\phi) + 2 d(x,X(\phi))$ holds also for elliptic elements and inversions. 

When $\phi \in \Hyp$ we denote by $a_{\phi}, r_{\phi} \in \partial T$ the two points of $\partial T$ corresponding to the two 
extremes of the axis $X(\phi)$. 
The point $a_{\phi}$, represented by the ray in the direction of translation, is called {\it{the attracting point of}} and 
$r_{\phi}$ is called its {\it repelling point}.
Note that if $a_{\phi} \in A, r_{\phi} \in R$ are open neighborhoods,
then for every large enough $n \in \mathbb{N}$ we have $\phi^n(\partial T \setminus R) \subset A$.
Since $\phi^{-1}$ and $\phi$ share the same axis, but translate in different directions, we have $a_{\phi^{-1}} = r_{\phi}, r_{\phi^{-1}} = a_{\phi}$.

If we replace the tree by its barycentric subdivision (adding one vertex in the middle of every geometric edge), we retain the same automorphism group, 
but every inversion becomes an elliptic element on the new tree\footnote{Recall that we have ruled out the case of tree consisting only of one bi-infinite geodesic.}. 
Using this trick we will assume from now on that there are no inversions in $\Aut(T)$. 

\begin{definition}
 A subgroup $G \le \Aut(T)$ is \emph{purely elliptic} (respectively, \emph{purely hyperbolic}) if all of its (non-trivial) elements are elliptic (respectively, hyperbolic).
\end{definition}

\subsection{Purely elliptic subgroups}
\begin{lemma}[Tits] \label{lem:Tits} (\cite[Proposition 26]{S03})
Let $\phi, \psi \in \Aut(T)$ be two elliptic elements with $X(\phi) \cap X(\psi) = \emptyset$. 
Then $\phi \psi, \psi \phi \in \Hyp$.  
\end{lemma}
\begin{proof}
Let $x_0,x_1,x_2,\ldots, x_m$ be the shortest geodesic from $X(\phi)$ to $X(\psi)$.
Let $e = (x_0,x_1)$ be the first directed edge along this path.
It is clear from Figure \ref{fig:Tits_lem} that $e$ and $\psi \phi e$ are co-oriented which proves the lemma. 
\begin{figure}[ht] 
\centering \def\svgwidth{300pt}
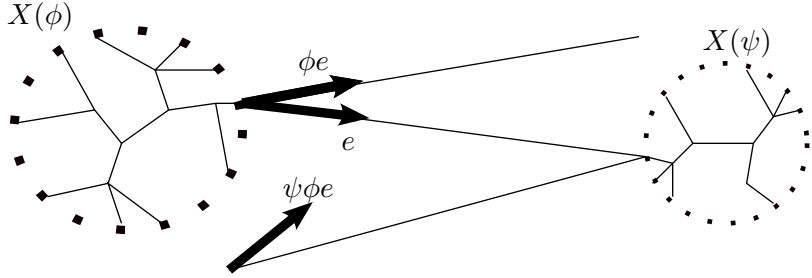 
\caption{Tits' lemma.}
\label{fig:Tits_lem}
\end{figure}

\end{proof}
The following is a Helly type lemma, capturing the fact that a tree is essentially a one-dimensional object. 
\begin{lemma} \label{lem:Helly}
If $X_1,X_2,X_3$ are convex subsets of $T$ with pairwise non-empty intersections then $X_1 \cap X_2 \cap X_3 \ne \emptyset$.
\end{lemma}
\begin{proof}
We refer to the indices modulo $3$.
Pick a point $y_{i,j} \in X_i \cap X_j$.
By convexity, the geodesic from $y_{i,i+1}$ to $y_{i,i+2}$ is contained in $X_i$.
Now, in any tree, every triangle has at least one point common to all of its edges (a tree is $0$-hyperbolic).
For the triangle with vertices $y_{i,j}$ this point will be in the desired intersection $X_1 \cap X_2 \cap X_3$. 
\end{proof}

\begin{corollary}[Classification of totally elliptic subgroups] \label{cor:pe}
Let $G \le \Aut(T)$ be a purely elliptic subgroup.
Then $G$ fixes a point in $T \cup \partial T$. 
If $G$ is finitely generated then it fixes a point in $T$. 
\end{corollary}
\begin{proof}
Assume first that $G$ is generated by a finite set $S$.
By Lemma \ref{lem:Tits},for any pair $s_i,s_j$ of generators, $X(s_i) \cap X(s_j) \ne \emptyset$.
Now, applying Lemma \ref{lem:Helly} successively, we find that $\bigcap_{s \in S} X(s) \ne \emptyset$, which proves the finitely generated case.

Assume that $G$ is purely elliptic but does not fix any point within the tree.
Given $x \in V(T)$ and $g \in G$ such that $gx \ne x$ we let $\sigma(x)$ be the neighbour of $x$ in the direction of $X(g)$. 
The point is that $\sigma x$ does not depend on $g$ as, by the finitely generated case, $X(g), X(h)$ are convex subtrees with a non-trivial intersection.
Consider the geodesic ray $\alpha_x := (x, \sigma x, \sigma^2 x, \ldots)$.
For two different vertices $x,y$ we have $d(\sigma(x), \sigma(y)) \le d(x,y)$ 
since we can choose some $g$ that fixes neither $x$ nor $y$ and $\sigma$ would then just be one step towards the convex set $X(g)$.
Thus $\alpha_x \sim \alpha_y$ represent the same point in the boundary. 
Moreover this boundary point is fixed by $G$ as $\alpha_x \sim \alpha_{gx} = g \alpha_x \in \partial T$. 
\end{proof}

\subsection{Geometric Density}
\begin{definition}\label{def:geom_dense}
A group $G \le \Aut(T)$ is called {\it{geometrically dense}} if it does not fix (pointwise) any end of $T$ and does not stabilize (as a set) any proper subtree of $T$. 
\end{definition}
This notion appears as a condition in Tits' simplicity theorem (our Theorem \ref{Tits-k}).
It can be generalized to much more general CAT(0) spaces, 
and in this capacity plays an important role in the work of Caprace--Monod \cite{CM09a,CM09b}
who argue that it should be thought of as a a geometric analogue of Zariski density.
For example, given a local field $k$ and considering the action of $\PGL_2(k)$ on its Bruhat--Tits tree, 
a subgroup $\Gamma < \PGL_2(k)$ is geometrically dense if and only if it is Zariski dense as a subgroup of $\PGL_2(K)$.
Here $K$ is the algebraic closure of $k$. 

\begin{lemma} \label{lem:min_tree}
If $G \le \Aut(T)$ contains at least one hyperbolic element
then there is a unique minimal $G$-invariant subtree of $T$. 
\end{lemma}
\begin{proof}
If $g \in G \cap \Hyp$ then its axis $X(g)$ is contained in every nonempty convex $G$-invariant set (refer to Figure \ref{fig:co_oriented}).
Let $Y \subset T$ be the smallest convex set containing the axes of all hyperbolic elements of $T$. 
Since $X(hgh^{-1}) = h X(g)$ the collection $\{X(g) \ | \ g \in G \cap \Hyp\}$ is invariant under $G$ and so is its convex hull $Y$.
Thus $Y$ is an invariant subtree which is contained in any other $G$-invariant tree.
\end{proof}
\begin{remark} \label{rem:denseL}
In the setting of the above lemma, let $L = \{a_g \ | \ g \in G \cap \Hyp\} \subset \partial T$ be the collection of attracting points of all the hyperbolic elements in $G$. 
As $L$ is $G$-invariant, so is its convex hull $\conv(L) \subset T \cup \partial T$. 
Thus $Y = \conv(L) \cap T$ is a $G$-invariant tree. 
Since $G$ contains a hyperbolic element $g$, the set $L$ contains at least two points $\{a_g, r_g = a_{g^{-1}} \}$ and $Y$ is non-empty.
If we further assume that $G$ is geometrically dense then $Y = T$ and consequently $L$ must be dense in $\partial T$. 
\end{remark}

\begin{lemma} \label{lem:gd_hyp}
Every geometrically dense subgroup $G \leq \Aut(T)$ contains a hyperbolic element. Furthermore, given any half-tree $Y \subset T$, the group $G$ contains a hyperbolic element whose axis is contained in $Y$.  
\end{lemma}
\begin{proof}
Recalling that we have assumed away the existence of inversions, the existence of hyperbolic elements follows directly from Corollary \ref{cor:pe}. By Remark \ref{rem:denseL} above for every half-tree $Y$, $\partial Y$ contains an attracting point $a_h$ for some $h \in \Hyp \cap G$.

Let $g \in \Hyp \cap G$ be another hyperbolic element with attracting and repelling fixed points $a_g,r_g$.
The conjugations $g_n := h^n g h^{-n}$ are again hyperbolic elements with $X(g_n) = h^n X(g), a_{g_n} = h^n a_g, r_{g_n} = h^n r_g$.
If $a_g \ne r_h, r_g \ne r_h$ then, due to the proximal action of $h$ on $\partial T$, 
all of these can be ``pushed'' arbitrarily close to $a_h$, and hence deep into the half-tree $Y$, by applying high powers of $h$. 
Thus it is enough to find such a $g$. 
We choose an element of the form $g = f h f^{-1}$ where $f \in G$. 
This is hyperbolic with attracting and repelling points $a_g = f a_h, r_g = fr_h$. 
Assume by way of contradiction  that $f^{-1} r_h \in \{r_h, a_h \} $for all  $f \in G$. 
Since $G$ is a group this means that either the point $\{r_h\}$ or the  set $\{r_h, a_h\}$ is $G$-invariant, 
contradicting our assumptions that $G$ is geometrically dense and that the tree is not an infinite line. 
\end{proof}

\begin{lemma} \label{lem:gd_normal}
Let $G \leq \Aut(T)$ be geometrically dense and assume that $N\leq \Aut(T)$ is a nontrivial group normalized by $H$.
Then $N$ is geometrically dense too. 
\end{lemma}
\begin{proof}
We start with the boundary. 
Let $A := \{\xi \in \partial T \ | \ n \xi = \xi \ \forall n \in N \} = X(N) \cap \partial T$.
We claim that $A$ is empty.
Clearly $A$ is closed and, since $N$ is normalized by $G$, the collection of its fixed points is stabilized by $G$. 
If $A$ consists of one point then this point is fixed by $G$, contradicting geometric density. 
If $|A| > 1$ let $Y$ be the convex hull of $A$ inside the tree $T$.
This is a non-empty $G$-invariant subtree of $T$ and, by the geometric density of $G$, we have $Y = T$.
This readily implies that $A$ is dense and hence $A = \partial T$, contradicting our assumption that $N$ is nontrivial.
Thus $N$ has no fixed point on the boundary.
A very similar argument shows that $N$ does not fix any point inside the tree. 

Since $N$ does not admit global fixed points in $T \cup \partial T$, Corollary \ref{cor:pe} implies that $N$ contains hyperbolic elements.
Thus, by Lemma \ref{lem:min_tree}, there is a unique minimal $N$-invariant subtree $Y \subset T$.
This unique tree is stabilized by $G$.
Since the latter is geometrically dense we actually have $Y = T$, which proves the theorem. 
\end{proof}

\section{Independence properties and Tits' simplicity theorem}\label{sec:Tits}
For any finite or (bi-)infinite path $C$ in $T$ and any $k \in\mathbb{N}$ 
 let $C^k$ denote the $k$-neighborhood of $C$ (i.e.\ the subtree of $T$ spanned by all vertices at distance at most $k$ from $C$). 
 Denote by $\pi_C: T \arrow C$ the nearest point projection on $C$ and let $T_x := \pi_C^{-1}(x) = \{z \in T \ | \forall y \in C: \ d(x,z) \le d(y,z) \}$.
 If $G\leq\Aut(T)$, then, for each vertex $x$ of $C$, 
 the pointwise stabilizer $\Fix_G(C^{k-1})$ acts on $T_x$. 
 Denoting by $F_x$ the permutation group induced by restricting $\mathrm{Fix}_G(C^{k-1})$ to $T_x$,
we obtain a map 
$$\Phi:\Fix_G(C^{k-1}) \rightarrow \prod_{x\in C} F_x$$ which is clearly an injective homomorphism.

 \begin{figure}[ht]\label{fig:Pk-def}
\centering
\def\svgwidth{\textwidth}
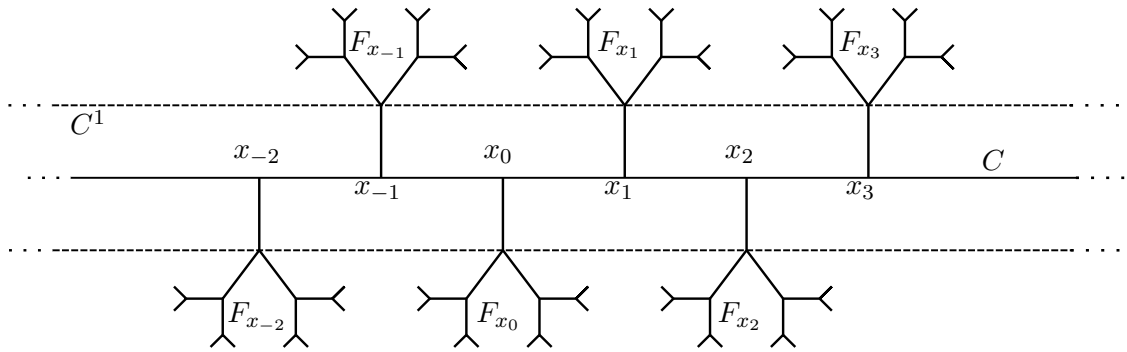
\caption{The restrictions $F_{x_i}$ of $\Fix_G(C^{1})$ to the subtrees $T_{x_i}$.}
\end{figure}
 
\begin{definition}
 We say that $G$ satisfies \emph{Property $P_k$} if for every finite or (bi-)infinite path $C$ in $T$ 
 the map $\Phi:\Fix_G(C^{k-1}) \rightarrow \prod_{x\in C} F_x$ defined above is an isomorphism.
 \end{definition}
 Notice that when $k=1$ we recover the original Property $P$ defined by Tits (\cite{T70}) so we
sometimes omit the subscript when referring to Property $P_1$.
We remark that Property $P$ is also known as Tit's Independence Property in the literature, 
because it ensures that the actions on subtrees rooted at a path can be chosen independently from each other.

To find simple subgroups of $\Aut(T)$ we will use a generalization of Tits' theorem (\cite[Th{\'e}or{\`e}me 4.5]{T70}). 

\begin{definition}
Let $$G^{+_k}:=\langle \Fix_G(e^{k-1})\mid e\in E\rangle$$
denote the subgroup of $G$ generated by pointwise stabilizers of ``$(k-1)$-thick'' edges.
In particular, $G^+:=G^{+_1}$ is generated by pointwise stabilizers of edges.
\end{definition}

\begin{theorem}\label{Tits-k}
 Suppose that $G\leq \mathrm{Aut}(T)$ is geometrically dense and satisfies Property $P_k$. 
 Then $G^{+_k}$ is either simple or trivial.
\end{theorem}

\begin{proof}
Figure \ref{fig:simplicity} illustrates the proof. 
\begin{figure}[ht] 
\centering \def\svgwidth{300pt}
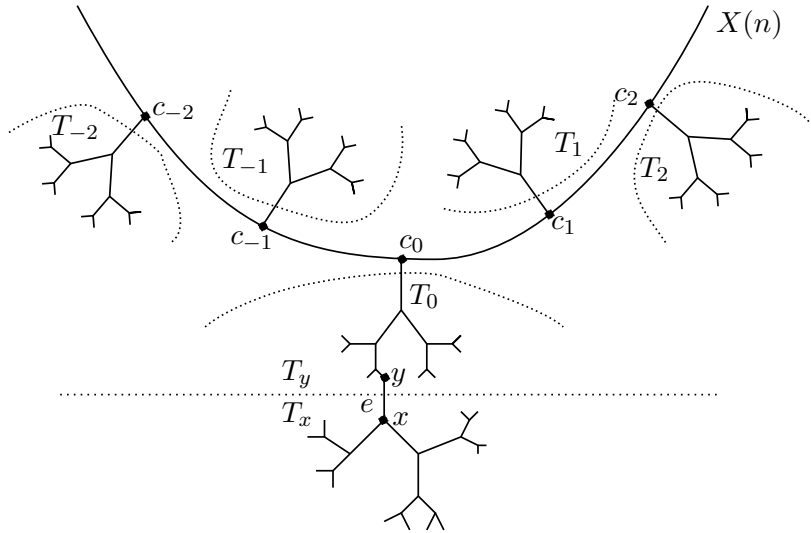 
\caption{Illustration of Tits' simplicity proof.}
\label{fig:simplicity}
\end{figure}

Write $H := G^{+_k}$. 
We assume that $H$ is nontrivial as otherwise there is nothing to prove.
Let $N\ne 1$ be normalized by $H$. 
We will show that $N \geq H$, thus proving the theorem.
Since $H$ is generated by the pointwise stabilizers $G_e(k) := \mathrm{Fix}_G(e^{k-1})$, 
it is enough to fix an edge $e \in E(T)$ and show that $N \geq G_e(k)$. 
Let $x,y$ be the vertices incident with $e$.
Writing $F := G_e(k)$, by Property $P_k$, there is a decomposition $F = F_x \times F_y$.
By the symmetry of the situation it is enough to show that $F_y \leq N$. 

By Lemma \ref{lem:gd_normal}  $H \lhd G$ is geometrically dense.
A second application of the same lemma shows that $N$ is geometrically dense too.
Lemma \ref{lem:gd_hyp} shows that there exists a hyperbolic element $n \in N$ with an axis whose $k$-neighbourhood is completely contained in the half-tree $T_x$.
We denote the axis by $C = X(n) =  (\ldots c_{-1}, c_{0},c_{1},c_{2} \ldots )$ with $n c_i = c_{i+\ell}$ where $\ell = \ell(n)$ is the translation length.
Assume that $c_0 = \pi_C(e)$ is the projection of the edge $e$ on this axis. 

We have to show that a given element $\phi \in F_y$ is contained in $N$.
Applying Property $P_k$ to the axis $C$ we obtain $\phi = \ldots \phi_{-1} \phi_0 \phi_1\phi_2 \ldots = \phi_0$. 
In the last equality we used the fact that $T_i = \pi_C^{-1}(c_i) \subset T_y$ for every $i \ne 0$ which implies that $\phi_i = \id$.
Since $N$ is normalized by $H$,  we have $[n,\psi] = n \psi n^{-1} \psi^{-1} \in N$ for all $\psi \in H$. 
To prove that $\phi \in N$ it is enough to exhibit some $\psi \in H$ such that $\phi = [n, \psi]$.
We will in fact find such an element in $\Fix_G(C^{k-1})$. 
If $\psi \in \Fix_{G}(C^{k-1})$ is an element we denote by $\psi_i = \psi_{T_i}$ its restriction to $T_i$. 
By property $P_k$ we have the freedom to construct $\psi$ by prescribing each element $\psi_i$ separately. 

With this notation we have $[n,\psi]_i = (n \psi n^{-1} \psi^{-1})_i = n|_{i-\ell} \circ \psi_{i-\ell} \circ (n^{-1})|_{i} \circ (\psi^{-1})_i$.
Solving for $[n,\psi] = \phi$ we obtain two equations
\begin{eqnarray*}
\psi_i & = & (\phi_i)^{-1} n|_{i-\ell} \circ \psi_{i-\ell} \circ (n|_{i-\ell})^{-1} \\
\psi_{i} & = & (n|_{i})^{-1} \circ \phi_{i+\ell} \circ \psi_{i+\ell} \circ n|_{i} 
\end{eqnarray*}
where for the second one we shifted all indices by $\ell$. 
Now assuming we have arbitrarily fixed the values of $\psi_0,\psi_1,\ldots,\psi_{\ell-1}$ we can now solve recursively for all other values of $\psi_{i}$ using the first of the above equations for the positive values of $i$ and the second one for negative values of $i$.
This method enables us to realize any element $\phi \in \Fix_G(C^{k-1})$ as a commutator of the form $[\psi,n]$.
Note that in our specific case we have $\phi_i = \id$ for all $i \ne 0$. This completes the proof. 
\end{proof}

\section{Universal Groups}\label{sec:universal}

In this section we introduce the universal group construction that was developed and studied by M. Burger and S. Mozes in \cite{BM00a}, 
and later on played an important role in the study of lattices in the product of the automorphism groups of two regular trees, see \cite{BM00b}.


Let $T_{d}=(V,E)$ be the $d$-regular tree ($d\in\mathbb{N},\ d\ge 3$). 
Recall that $E(x)$ denotes the set of edges with origin $x\in V$.
Further, let $l:E\to\{1,\ldots,d\}$ be a legal labelling of $T_{d}$, i.e.\ for every $x\in V$ the map
\begin{displaymath}
  l_{x}:E(x)\to\{1,\ldots,d\},\ y\mapsto l(y)
\end{displaymath}
is a bijection, and $l(y)=l(\overline{y})$ for all $y\in E$. A ball of radius two in $T_{3}$ looks as follows:

\begin{figure}[ht]
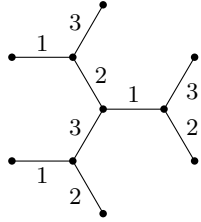

{\footnotesize
\begin{displaymath}
  \psset{unit=0.8cm}
  \psset{labelsep=2.5pt}
  \psset{linewidth=0.4pt}
  \pspicture(-2,-1.75)(2,2)
  \psdots[](0,0)
    \psline(0,0)(1,0)
    \psline(0,0)(-0.5,0.86)
    \psline(0,0)(-0.5,-0.86)
    \psdots[](1,0)
    \psdots[](-0.5,0.86)
    \psdots[](-0.5,-0.86)
    \uput[u](0.5,0){$1$}
    \uput[30](-0.25,0.43){$2$}
    \uput[150](-0.25,-0.43){$3$}
    \psline(1,0)(1.5,-0.86)
    \psline(1,0)(1.5,0.86)
    \psdots[](1.5,-0.86)
    \psdots[](1.5,0.86)
    \uput[30](1.25,-0.43){$2$}
    \uput[330](1.25,0.43){$3$}

    \psline(-0.5,0.86)(-1.5,0.86)
    \psline(-0.5,0.86)(0,1.73)
    \psdots[](-1.5,0.86)
    \psdots[](0,1.73)
    \uput[150](-0.25,1.30){$3$}
    \uput[u](-1,0.86){$1$}

    \psline(-0.5,-0.86)(-1.5,-0.86)
    \psline(-0.5,-0.86)(0,-1.73)
    \psdots[](-1.5,-0.86)
    \psdots[](0,-1.73)
    \uput[210](-0.25,-1.30){$2$}
    \uput[d](-1,-0.86){$1$}
  \endpspicture
\end{displaymath}}
\caption{A legally labelled ball of radius two in $T_{3}$.}
\end{figure}

Now, given $x\in V$, every automorphism $g\in\Aut(T_{d})$ induces a permutation at $x$ via the following map:
\begin{displaymath}
 c:\Aut(T_{d})\times V\to S_{d},\ (g,x)\mapsto l_{gx}\circ g\circ l_{x}^{-1}.
\end{displaymath}

\begin{definition}\label{def:universal_groups}
Let $F\le S_{d}$. Define $\mathrm{U}^{(l)}(F):=\{g\in\Aut(T_{d})\mid \forall x\in V:\ c(g,x)\in F\}$.
\end{definition}

It is immediate from the following cocycle property of the map $c$ that the sets introduced in Definition \ref{def:universal_groups} are in fact groups.

\begin{lemma}\label{lem:cocycle_identity}
Let $x\in V$ and $g,h\in\Aut(T_{d})$. Then $c(gh,x)=c(g,hx)c(h,x)$. \qed
\end{lemma}

The groups of Definition \ref{def:universal_groups} are termed \emph{universal groups} because of Proposition \ref{prop:universality} below. In words, they consist of those automorphisms of the regular tree which around every vertex act like one of the allowed permutations from $F$. 
To determine the dependence of $\mathrm{U}^{(l)}(F)$ on the labelling $l$, we record the following.

\begin{lemma}\label{lem:ext}
Given a quadruple $(l,l',b,b')$ consisting of legal labellings $l,l'$ of $T_{d}$ and vertices $b,b'\in V$, there is a unique automorphism $g\in\Aut(T_{d})$ with $gb=b'$ and $l'=l\circ g$.
\end{lemma}

\begin{proof}
By assumption, $gb=b'$. 
Now assume inductively that $g$ is uniquely determined on $B(b,n)$ $(n\in\mathbb{N}_{0})$ and let $x\in V$ be at distance $n$ from $b$.
Then $g$ is also uniquely determined on $E(x)$ by the requirement that $l'=l\circ g$, namely $g|_{E(x)}:=l|_{E(gx)}^{-1}\circ l'|_{E(x)}$.
\end{proof}

\begin{corollary}\label{cor:labellings}
Let $l$ and $l'$ be legal labellings of $T_{d}$. Further, let $F\le S_{d}$. Then $\smash{\mathrm{U}^{(l)}(F)}$ and $\smash{\mathrm{U}^{(l')}(F)}$ are conjugate in $\Aut(T_{d})$.
\end{corollary}

\begin{proof}
Choose $b\in V$. If $\tau\in\Aut(T_{d})$ is the automorphism of $T_{d}$ associated to $(l,l',b,b)$ by Lemma \ref{lem:ext}, then  $\smash{\mathrm{U}^{(l)}(F)=\tau\mathrm{U}^{(l')}(F)\tau^{-1}}$.
\end{proof}

With Corollary \ref{cor:labellings} in mind, we henceforth omit the reference to an explicit labelling.

\begin{example}\label{ex:uf}
Clearly, $\mathrm{U}(S_{d})=\Aut(T_{d})$. 
On the other hand, $\mathrm{U}(\{\mathrm{id}\})\cong\mathbb{Z}/2\mathbb{Z}\ast\cdots\ast\mathbb{Z}/2\mathbb{Z}$ where the number of copies of $\mathbb{Z}/2\mathbb{Z}$ is $d$.
To see this, fix $b\in V$ and for $i\in\{1,\ldots,d\}$ denote by $e_{i}\in E$ the edge with $o(e_{i})=b$ and $l(e_{i})=i$.
Further, let $\sigma_{i}\in\mathrm{U}(\{\mathrm{id}\})$ denote the unique label-respecting inversion of the edge $e_{i}$, which is associated to $(l,l,o(e_{i}),t(e_{i}))$ by Lemma \ref{lem:ext}.
Then the subgroups $\langle\sigma_{1}\rangle,\ldots,\langle\sigma_{d}\rangle$ generate the asserted free product within $\mathrm{U}(\{\mathrm{id}\})$ by an application of the ping-pong lemma.
Finally, every $\alpha\in\mathrm{U}(\{\mathrm{id}\})$ is the unique automorphism of $T_{d}$ associated to $(l,l,b,\alpha(b))$ by Lemma \ref{lem:ext}
which can be realized as an element of $\langle\sigma_{1}\rangle\ast\cdots\ast\langle\sigma_{d}\rangle$ 
by composing the inversions along the edges that occur in the unique reduced path from $b$ to $\alpha(b)$.
\end{example}

Lemma \ref{lem:ext} also plays an important role in proving the following list of basic properties of the groups $\mathrm{U}(F)$. 

\begin{proposition}\label{prop:uf_properties}
Let $F\le S_{d}$. Then the following statements hold.
\begin{itemize}
 \item[(i)] $\mathrm{U}(F)$ is closed in $\mathrm{Aut}(T_{d})$.
 \item[(ii)] $\mathrm{U}(F)$ is locally permutation isomorphic to $F$.
 \item[(iii)] $\mathrm{U}(F)$ is vertex-transitive.
 \item[(iv)] $\mathrm{U}(F)$ is edge-transitive if and only if the action $F\curvearrowright\{1,\ldots,d\}$ is transitive.
 \item[(v)] $\mathrm{U}(F)$ is discrete in $\Aut(T_{d})$ if and only if the action $F\curvearrowright\{1,\ldots,d\}$ is free.
\end{itemize}
\end{proposition}

Proposition \ref{prop:uf_properties} illustrates the principle that properties of $\mathrm{U}(F)$ correspond to properties of $F$, which constitutes part of the beauty of the universal group construction.

\begin{proof}
For (i), suppose that $g\in\Aut(T_{d})\setminus\mathrm{U}(F)$. Then $c(g,x)\not\in F$ for some $x\in V$ and hence the open neighbourhood $\{h\in\Aut(T_{d})\mid h|_{B(x,1)}=g|_{B(x,1)}\}$ of $g$ is also contained in the complement of $\mathrm{U}(F)$ in $\Aut(T_{d})$.

For (ii), let $b\in V$ and $a\in F$. Further, let $\alpha\in\Aut(T_{d})$ be the automorphism associated to $(l,a\circ l,b,b)$. 
Then $c(\alpha,x)=a$ for every $x\in V$ and hence $\alpha\in\St_{\mathrm{U}(F)}(b)$ realizes the permutation $a$ at the vertex $b$.

For part (iii), let $b,b'\in V$ and let $g\in\Aut(T_{d})$ be the automorphism of $T_{d}$ associated to $(l,l,b,b')$ by Lemma \ref{lem:ext}. Then $g\in\mathrm{U}_{k}(F)$ as $c(g,x)=\id\in F$ for all $x\in V$.

As to (iv), suppose that $F$ is transitive. Given $e,e'\in E$, choose $\alpha'\in\mathrm{U}(F)$ such that $\alpha'o(e)=o(e')$ by (iii). Then pick $\alpha''\in\mathrm{U}(F)_{o(e')}$ such that $\alpha''(\alpha'e)=e'$, by (ii) and transitivity of $F$, and set $\alpha:=\alpha''\circ\alpha'$.

Conversely, if $\mathrm{U}(F)$ is edge-transitive then $\mathrm{U}(F)_{x}$ acts transitively on $E(x)$ for a given vertex $x\in V$ and hence $F$ is transitive by (ii).

For part (v), fix $b\in V$ and suppose that the action $F\curvearrowright\{1,\ldots,d\}$ is not free, say $a\in F-\{e\}$ fixes $i\in\{1,\ldots,d\}$. For every $n\in\mathbb{N}$ we define $\alpha\in\mathrm{U}(F)$ such that $\alpha|_{B(b,n)}=\id$ but $\alpha|_{B(b,n+1)}\neq\id$ as follows: Set $\alpha|_{B(b,n)}:=\id$ and let $e\in E$ be an edge with $o(e)\in S(b,n)$, $t(e)\in S(b,n-1)$ and $l(e)=i$. Then we may extend $\alpha$ to $B(b,n+1)$ as desired by setting $\alpha|_{E(x)}:=l_{x}^{-1}\circ a\circ l_{x}$ and $\alpha|_{E(x')}:=\id$ for all $x'\in S(b,n)-x$. Now, inductively extend $\alpha$ to $T_{d}$ such that $c(\alpha,x')=a$ whenever $x'\in V$ maps to $x$ under the projection onto $B(x,n)$ and $c(\alpha,x')=\id$ otherwise.

Conversely, assume that $\mathrm{U}(F)$ is non-discrete. Then there are $\alpha\in\mathrm{U}(F)$ and $n\in\mathbb{N}$ such that $\alpha|_{B(b,n)}=\id$ but $\alpha|_{B(b,n+1)}\neq\id$. Hence there is $x\in S(b,n)$ such that $c(\alpha,x)\in F$ is non-trivial and fixes a point.
\end{proof}

For the sake of clarity we extract the following statement from Proposition \ref{prop:uf_properties}.

\begin{proposition}
Let $F\le S_{d}$. Then $\mathrm{U}(F)$ is compactly generated, totally disconnected, locally compact Hausdorff. It is discrete if and only if the action $F\curvearrowright\{1,\ldots,d\}$ is free.
\end{proposition}

\begin{proof}
The group $\mathrm{U}(F)$ is totally disconnected, locally compact Hausdorff as a closed subgroup of $\Aut(T_{d})$. Furthermore, $\mathrm{U}(F)$ is generated by the compact set $\mathrm{U}(F)_{b}\cup\{\sigma_{1},\ldots,\sigma_{d}\}$ where $b\in V$ is a fixed vertex and $\sigma_{i}$ is the edge-inversion of Example \ref{ex:uf}. This follows from vertex-transitivity of $\mathrm{U}(\{\mathrm{id}\})=\langle\sigma_{1}\rangle\ast\cdots\ast\langle\sigma_{d}\rangle$: For $\alpha\in\mathrm{U}(F)$ pick $\beta\in\mathrm{U}(\{\mathrm{id}\})$ such that $\beta(\alpha b)=b$. Then $\beta\alpha\in\mathrm{U}(F)_{b}$, hence the assertion.
\end{proof}

\subsection{Simplicity}
Since $\mathrm{U}(F)$ is vertex transitive it cannot stabilize any subtree. Since it is transitive on the (directed) edges of any given color it is easy to see that it cannot fix any end of the tree. Thus $\mathrm{U}(F)$ is geometrically dense.
We claim that the group satisfies Property $P$. For every $e=(x,y)\in E$ we have
\begin{displaymath}
 \mathrm{U}(F)_{e}\xrightarrow{\cong}\mathrm{U}(F)_{T_{y}}\times\mathrm{U}(F)_{T_{x}},\ \alpha\mapsto(\alpha_{x},\alpha_{y})
\end{displaymath}
where $\alpha_{x}$ is given by $\alpha$ on $T_{x}$ and the identity elsewhere, and similarly for $\alpha_{y}$. Then for $x'\in T_{x}$ we have $c(\alpha_{x},x')=c(\alpha,x')\in F$ and $c(\alpha_{x},-)=\id$ otherwise. This argument carries over to arbitrary finite or infinite paths in $T_{d}$.

\begin{figure}[ht]
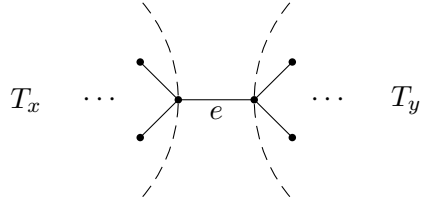

\begin{displaymath}
 \psset{unit=1cm}
 \psset{labelsep=2.5pt}
 \psset{linewidth=0.4pt}
 \pspicture(-3,-1.5)(3,1.5)
  \psdots[](-0.5,0)
  \psline(-0.5,0)(0.5,0)
  \uput[d](0,0){$e$}
  
  \psline(-0.5,0)(-1,0.5)
  \psline(-0.5,0)(-1,-0.5)
  \psdots[](0.5,0)
  \psdots[](-1,0.5)
  \psdots[](-1,-0.5)
  \rput(-1.5,0){$\cdots$}
  \rput(-2.5,0){$T_{x}$}
 
  \psline(0.5,0)(1,-0.5)
  \psline(0.5,0)(1,0.5)
  \psdots[](1,-0.5)
  \psdots[](1,0.5)
  \rput(1.5,0){$\cdots$}
  \rput(2.5,0){$T_{y}$}

  \psarc[linestyle=dashed](-2.5,0){2}{320}{40}
  \psarc[linestyle=dashed](2.5,0){2}{140}{220}
 \endpspicture
\end{displaymath}
\caption{Illustration of Tits' simplicity criterion.}
\end{figure}

As a consequence, the subgroup $\mathrm{U}(F)^{+}:=\{g\in\mathrm{U}(F)\mid \exists e\in E: ge=e\}$ of $\mathrm{U}(F)$ is simple in many cases. More precisely, we have the following.

\begin{theorem}
Let $F\le S_{d}$. Then $\mathrm{U}(F)^{+}$ is either trivial or simple. If $F$ is transitive and generated by point stabilizers we have $\mathrm{U}(F)^{+}=\mathrm{U}(F)\cap\Aut(T_{d})^{+}$ and therefore $[\mathrm{U}(F):\mathrm{U}(F)^{+}]=2$.
\end{theorem}

\begin{proof}
The simplicity assertion follows from Theorem \ref{Tits-k}.
If $F$ is transitive and generated by edge stabilizers then $\mathrm{U}(F)^{+}$ acts transitively on geometric edges. This implies the non-trivial inclusion in the equality $\mathrm{U}(F)^{+}=\mathrm{U}(F)\cap\Aut(T_{d})^{+}$: Namely, let $g\in\mathrm{U}(F)\cap\Aut(T_{d})^{+}$ and fix $e\in E$. Choose $h\in\mathrm{U}(F)^{+}$ such that $h(ge)=e$. Then $hg\in\mathrm{U}(F)^{+}$ and hence $g\in\mathrm{U}(F)^{+}$. Since $[\Aut(T_{d}):\Aut(T_{d})^{+}]=2$, the non-trivial coset being given by an edge-inversion, this implies $[\mathrm{U}(F):\mathrm{U}(F)^{+}]\le 2$ and equality follows from, say, the existence of edge-inversions in $\mathrm{U}(F)$.
\end{proof}

\subsection{Universality}
The groups $\mathrm{U}(F)$ are universal in the following sense.

\begin{proposition}\label{prop:universality}
Let $H\le\Aut(T_{d})$ be vertex-transitive and locally permutation isomorphic to $F\le S_{d}$. If $F$ is transitive then there is a legal labelling $l$ of $T_{d}$ such that $H\le\mathrm{U}^{(l)}(F)$.
\end{proposition}

\begin{proof}
Fix $b\in V$. Since $H$ is locally permutation isomorphic to $F$, there is a bijection $l_{b}:E(b)\to\{1,\ldots,d\}$ such that $H_{b}|_{E(b)}=l_{b}^{-1}\circ F\circ l_{b}$. We now inductively define a legal labelling $l:E\to\{1,\ldots,d\}$ such that $\smash{H\le\mathrm{U}^{(l)}(F)}$. Set $l|_{E(b)}:=l_{b}$ and suppose inductively that $l$ is defined on $\smash{E(b,n):=\bigcup_{x\in B(b,n-1)}E(x)}$. To extend $l$ to $E(b,n+1)$, let $x\in S(b,n)$ and let $e_{x}\in E$ be the unique edge with $o(e_{x})=x$ and $d(b,t(e_{x}))+1=d(b,x)$. Since $H$ is vertex-transitive and locally permutation isomorphic to the transitive group $F$, there is an element $\sigma_{e_{x}}\in H$ which inverts the edge $e_{x}$. We may thus legally extend $l$ to $E(x)$ by setting $l|_{E(x)}:=l\circ\sigma_{e_{x}}$.

To check the inclusion $\smash{H\le\mathrm{U}^{(l)}(F)}$, let $x\in V$ and $h\in H$. If $(b,b_{1},\ldots,b_{n},x)$ and $(b,b_{1}',\ldots,b_{m}',h(x))$ denote the unique reduced paths from $b$ to $x$ and $h(x)$, then
\begin{displaymath}
 s:=\sigma_{e_{b_{1}'}}\cdots\sigma_{e_{b_{m}'}}\sigma_{e_{h(x)}}\circ h\circ\sigma_{x}\sigma_{e_{b_{n}}}\cdots\sigma_{e_{b_{2}}}\sigma_{e_{b_{1}}}\in H_{b}
\end{displaymath}
and we have $c(h,x)=c(s,b)\in F$ by Lemma \ref{lem:cocycle_identity}.
\end{proof}

\subsection{Structure of a Point Stabilizer} In this section, we exhibit a point stabilizer in $\mathrm{U}(F)$ as a profinite group in terms of $F$ for transitive $F\le S_{d}$. To this end, let $b\in V$, $\Delta:=\{1,\ldots,d\}$, $D:=\{1,\ldots,d-1\}$ and set $\Delta_{n}:=\Delta\times D^{n-1}$. We fix bijections $b_{n}:S(b,n)\to\Delta_{n}$ as follows: Given that $F$ is transitive we may for every $i\in\{1,\ldots,d\}$ fix an element $a_{i}\in F$ with $a_{i}(i)=d$. Define inductively
\begin{itemize}
 \item[(i)] $b_{1}:S(b,1)\to\Delta_{1},\ x\mapsto l((b,x))$, and
 \item[(ii)] $b_{n+1}:S(b,n+1)\to\Delta_{n+1}=\Delta_{n}\times D,\ x\mapsto (b_{n}p_{n}x,a_{l(p_{n-1}x,p_{n}x)}(l(p_{n}x),x))$
\end{itemize}
for $n\in\mathbb{N}$, where $p_{n}:\bigcup_{k\ge n}S(b,k)\to S(b,n)$ is the canonical projection. We now capture the action of $\mathrm{U}(F)_{b}$ on $S(b,n)$ by inductively defining $F(n)\le\mathrm{Sym}(\Delta_{n})$ as follows: Let $F_{d}:=\mathrm{stab}_{F}(d)$, set
\begin{itemize}
 \item[(i)] $F(1):=F\le\mathrm{Sym}(\Delta_{1})$, and define
 \item[(ii)] $F(n+1):=F(n)\ltimes F_{d}^{\Delta_{n+1}}\le\mathrm{Sym}(\Delta_{n})$
\end{itemize}
to be the wreath product for the action of $F(n)$ on $\Delta_{n}$. Further, let $\pi_{n}:F(n)\to F(n-1)$ denote the canonical projection. The bijection $b_{n}$ induces the surjective homomorphism \begin{displaymath}
  \varphi_{n}:\mathrm{U}(F)_{b}\to F(n)\le\mathrm{Sym}(\Delta_{n}),\ g\mapsto b_{n}\circ g\circ b_{n}^{-1}
\end{displaymath}
with kernel $\{g\in\mathrm{U}(F)_{b}\mid g|_{S(b,n)}=\id\}=\{g\in\mathrm{U}(F)_{b}\mid g|_{B(b,n)}=\id\}$ and one readily checks that the map
\begin{displaymath}
 \varphi:=(\varphi_{n})_{n\in\mathbb{N}}:\mathrm{U}(F)_{b}\to \underset{\longleftharpoondown}{\lim}\ F(n)=\left\{\left.(f_{n})_{n=1}^{\infty}\in\prod_{n=1}^{\infty}F(n)\right|\forall n\in\mathbb{N}:\pi_{n}f_{n}=f_{n-1}\right\}
\end{displaymath}
is an isomorphism of topological groups: Abbreviate $G:=\smash{\underset{\longleftharpoondown}{\lim}\ F(n)}$. 
Clearly, $\varphi$ is a bijective homomorphism.
To prove that it is a homeomorphism, note that $\mathrm{U}(F)_{b}$ is compact and $\smash{\underset{\longleftharpoondown}{\lim}\ F(n)}$ is Hausdorff; therefore $\varphi$ is closed and it suffices to show continuity. Let $U:=G\cap\prod_{n=1}^{\infty}U_{n}$ be a basic open neighbourhood of $\smash{f\in\underset{\longleftharpoondown}{\lim}\ F(n)}$. Then there is $N\in\mathbb{N}$ such that $U_{n}=F(n)$ for all $n\ge N$ and hence for every $g\in\varphi^{-1}(U)$ the open neighbourhood $\{h\in\mathrm{U}(F)_{b}\mid h|_{B(b,N)}=g|_{B(b,N)}\}$ of $g$ is contained in $\varphi^{-1}(U)$.

\section{Simple totally disconnected locally compact groups with prescribed local actions}\label{sec:k-closure}
The purpose of this section is to find new examples of t.d.l.c.\ groups which are (abstractly) simple, compactly generated and non-discrete. 
The motivation for this comes from classification results of locally compact groups by Caprace and Monod \cite{CM11}
which yield cases in which 
compactly generated t.d.l.c.\ groups decompose into (topologically) simple compactly generated non-discrete pieces. 
We are still at the stage of collecting examples of such simple groups,
aiming in the long term for some sort of classification.
The examples collected so far can be classified into the following types:
\begin{itemize}
 \item simple Lie groups
 \item simple algebraic groups over local fields
 \item complete Kac--Moody groups over finite fields
 \item automorphism groups of trees, some CAT(0) cube-complexes, and right-angled buildings
 \item variations on the above (e.g.\ almost automorphisms of trees).
\end{itemize}

To obtain these new examples of compactly generated, simple, non-discrete t.d.l.c.\ groups, 
we shall introduce some ``$k$-thickened'' (for $k\in\mathbb{N}$) variations of the universal group construction:
the \emph{$k$-closures} of a given $G\leq \Aut(T)$.
These prescribe the action on all balls of radius $k$ by elements of $G$.
We will then see that the $k$-closure of a group of tree automorphisms satisfies Property $P_k$
and use Theorem \ref{Tits-k} to obtain abstractly simple t.d.l.c.\ groups which are compactly generated. 
The last step will be ensuring that these are non-discrete and different.

\subsection{$k$-closures and Property $P_k$}

\begin{definition}
 Let $G\leq \mathrm{Aut}(T)$ and $k\in \mathbb{N}$. 
 The \emph{k-closure} of $G$ is 
 $$G^{(k)}:=\{ h\in \mathrm{Aut}(T)\mid \forall x\in V(T): \exists g\in G: h|_{B(x,k)} = g|_{B(x,k)} \},$$ all automorphisms of $T$ that agree on each ball of radius $k$ with some element of $G$.
\end{definition}
In this setting, $G$ is  the analogue of $F$ in the definition of $\mathrm{U}(F)$, providing a list of ``allowed'' actions.
However, here we do not require that the tree be regular.
Note that a given $h\in G^{(k)}$ need not agree with the same element of $G$ on every ball;
the point is that for each ball there is \emph{some} element of $G$ agreeing with $h$, and they may all be different for each ball. The $k$-closure of $G$ has the following basic properties, which justify the term ``closure''.

\begin{proposition}\label{A:closure}
Let $G\le\Aut(T)$ and $k\in\mathbb{N}$.
 \begin{itemize}
  \item[(i)] $G^{(k)}$ is a closed subgroup of $\mathrm{Aut}(T)$.
  \item[(ii)] For every $k,l\in \mathbb{N}$ with $l>k$ we have $G\leq G^{(l)}\leq G^{(k)}$.
  \item[(iii)] We have $\bigcap_{k\in\mathbb{N}} G^{(k)}=\overline{G}$ (the topological closure of $G$ in $\mathrm{Aut}(T)$).
 \end{itemize}
\end{proposition}
\begin{proof} \hspace{0cm}
 \begin{itemize}
  \item[(i)] To see that $G^{(k)}$ is indeed a subgroup of $\Aut(T)$, notice that if $a, b \in G^{(k)}$ and $x\in V(T)$ then 
  there exist $g, h\in G$ such that $a|_{B(x,k)}=g|_{B(x,k)}$ and $b|_{B(ax,k)}=h|_{B(ax,k)}$, so 
  $b\circ a|_{B(x,k)}=h\circ g|_{B(x,k)}$ and $b\circ a\in G^{(k)}$.
  Also, there exists some $f\in G$ such that $a|_{B(a^{-1}x,k)}=f|_{B(a^{-1}x,k)}$, so $a^{-1}|_{B(x,k)}=f^{-1}|_{B(x,k)}$ 
  and $a^{-1}\in G^{(k)}$.
  
  For the closure part, note that for each $a\notin G^{(k)}$ there is some vertex $x_a$ such that no element of $G$ agrees with $a$ on $B(x_a,k)$. 
  Thus $\Aut(T)\setminus G^{(k)} = \bigcap_{a\notin G^{(k)}} \mathcal{U}(a, B(x_a,k))$ is a union of basic open sets.
 
  \item[(ii)] The group $G$ agrees with itself on balls of all radii so $G\leq G^{(l)}$ for all $l$
  and if $l>k$ then $G^{(l)}$ certainly agrees with $G$ on balls of smaller radius $k$, so $G^{(l)}\leq G^{(k)}$.
  
  \item[(iii)] Since $G\leq G^{(k)}$ for all $k$ and $G^{(k)}$ is closed, the closure of $G$ must be contained in every $G^{(k)}$ 
  and therefore in the intersection of all of them.
  For the other direction, we show that every element $a\in \bigcap_{k\in \mathbb{N}} G^{(k)}$ is a point of closure of $G$. 
  Fix a vertex $x$ and consider $\mathcal{U}(a,B(x,k))$; 
  then, since $a \in G^{(k)}$, there is some $g\in G$ such that $g\in \mathcal{U}(a,B(x,k))$.
\end{itemize}
\end{proof}
Just as U$(F)$ satisfies Property $P$
the $k$-closure of $G$ satisfies Property $P_k$.
 
\begin{proposition}\label{A:prop}
Let $G\leq \mathrm{Aut}(T)$ and $k\in\mathbb{N}$. Then $G^{(k)}$ satisfies Property $P_k$.
\end{proposition}
\begin{proof}
Let $C=(\dots,x_0,x_1,\dots,x_n,\dots)$ be any finite or (bi-)infinite path
and suppose that $f:=(\dots,f_0,f_1,\dots,f_n,\dots) \in \prod_{x_i\in C}F_{x_i}$.
To see that $f\in G^{(k)}$, pick a vertex $v$, which must be in $T_{x_i}$ for some $x_i\in C$.
By definition, $f_i$ is the restriction to $T_{x_i}$ of some $h_i\in\Fix_{G^{(k)}}(C^{k-1})$. 
Thus, if $B(v,k)$ is entirely contained in $T_{x_i}$ then $f|_{B(v,k)}=h_i|_{B(v,k)}=g|_{B(v,k)}$ for some $g\in G$, since $h_i\in G^{(k)}$.
And if there is some part of $B(v,k)$ outside $T_{x_i}$ then both $f$ and $h_i$ act trivially on it. 
In either case, there is some $g\in G$ such that $f|_{B(v,k)}=g|_{B(v,k)}$ and $f\in G^{(k)}$.
\end{proof}

Satisfying Property $P_k$ characterizes when the process of taking $k$-closures stabilizes.
\begin{theorem}
 If $G\leq \mathrm{Aut}(T)$ satisfies Property $P_k$ then $G^{(k)}=\overline{G}$ 
 (and if $G^{(k)}=\overline{G}$ then $\overline{G}$ satisfies $P_k$ by Proposition \ref{A:prop}).
\end{theorem}

 \begin{figure}[ht]\label{A:Pk_stableclosures}
 \tikzstyle{level 1}=[sibling angle=120]
\tikzstyle{level 2}=[sibling angle=60]

\begin{tikzpicture}[grow cyclic,shape=circle,
                    cap=round]
\node[draw, level distance=5mm] {$v$} 
	child [label=below:\A] foreach \A in {1,2,m} 
			{ node[draw, inner sep=1pt] {$u_{\A}$} 			
				child[level distance=22mm] { node {} edge from parent[snake=expanding waves,segment length=2mm,segment angle=60, dashed, draw]node[fill=black!20]{$b_{\A}$} } 
    };
\end{tikzpicture}
\caption{Each of the automorphisms $b_i$ acts on the subtree rooted at $u_i$.}
\end{figure}
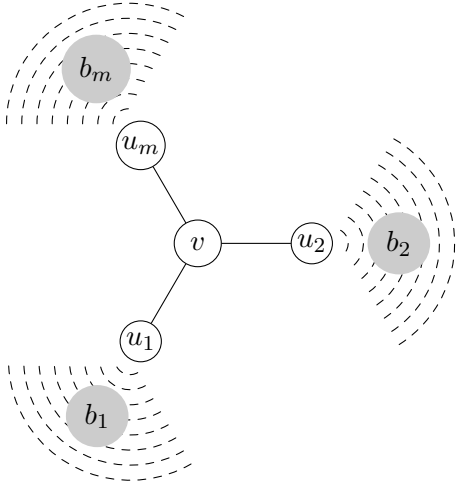

\begin{proof}
 We know from \ref{A:closure} that $\overline{G}=\bigcap_{k\in\mathbb{N}} G^{(k)}$,
 so it suffices to show that $G^{(k)}=G^{(n)}$ for all $n\geq k$.
 To illustrate the proof, we only show the case $G^{(1)}=G^{(2)}$ (see \cite[Theorem 5.4]{BEW11} for the full proof).
 Let $x\in G^{(1)}$.
 For each vertex $v$ there is some $g\in G$ such that $x|_{B(v,1)}=g|_{B(v,1)}$; thus $xg^{-1}\in \Fix_{G^{(1)}}(B(v,1))$. 
 Suppose that $u_1,\dots,u_m$ are the neighbours of $v$.
 Since $xg^{-1}\in G^{(1)}$, for each $i$ there exist $a_i\in G$ such that $xg^{-1}|_{B(u_i,1)}=a_i|_{B(u_i,1)}$.
 So $a_i$ fixes the edge $(v,u_i)$.
 Because $G$ satisfies property $P_1$, there exist unique $b_i, c_i\in G$
 such that $a_i=b_ic_i$,
 where $b_i$ only acts non-trivially on $T_{u_i}$ (and $c_i$ fixes $T_{u_i}$, see Figure \ref{A:Pk_stableclosures}). 
 Then the product $b_1\dots b_m$ fixes all neighbours of $v$ and $v$ itself; that is,  it fixes $B(v,1)$.
 Furthermore, $c_i$ fixes $T_{u_i}$, so $b_i|_{B(u_i,1)}=xg^{-1}|_{B(u_i,1)}$ and hence $b_1\dots b_m|_{B(u_i,1)}=xg^{-1}|_{B(u_i,1)}$ for each $i$. 
 Thus $b_1\dots b_m|_{B(v,2)}=xg^{-1}|_{B(v,2)}$ and $b_1\dots b_mg|_{B(v,2)}=x{-1}|_{B(v,2)}$. 
 Since $b_1\dots b_mg\in G$ we conclude that $x\in G^{(2)}$, as required. 
\end{proof}

More importantly, we deduce the following statement which will be used to find infinitely many distinct simple subgroups.
\begin{corollary}\label{stable_k_closures}
 There are infinitely many distinct $k$-closures of $G$ if and only if $\overline{G}$ does not satisfy Property $P_k$ for any $k$.
\end{corollary}
\begin{proof}
 If $\overline{G}$ does not satisfy Property $P_k$ for any $k$, then $G^{(k)}\neq \overline{G}$ (by Proposition \ref{A:prop}). 
 Hence $G^{(k)}\neq \bigcap_n G^{(n)}$ for all $k$ and therefore the sequence $(G^{(k)})_k$ never stabilizes; 
 in particular there are infinitely many distinct $k$-closures $G^{(k)}$ of $G$. 
 For the converse, we have that $(G^{(k)})_k$ never stabilizes, therefore there is no $k$ such that $G^{(k)}=\bigcap_n G^{(n)}=\overline{G}$
 and so $G$ does not satisfy $P_k$ for any $k$. 
\end{proof}

\subsection{Local rigidity for $k$-closures}
We digress a moment from our objective in this section of finding infinitely many simple groups, to point out a local-global result of Burger--Mozes  that is relevant to $k$-closures.
\begin{theorem}{\cite[Proposition 3.3.1]{BM00a}}
Let $F \leq \Sym(d)$ be a finite $2$-transitive permutation group on the set $\{1,2,\ldots,d\}$ and 
$F_1$ the stabilizer of a point under this action. Assume that $F_1$ is non-abelian and simple. 

Let $T$ be a $d$-regular tree and $G \leq \Aut(T)$ a vertex-transitive subgroup. 
If $x \in V(T)$ is any vertex we have a map $$\phi: \St_G(x) \arrow \Sym(d)$$ given by the action of $\St_G(x)$ on $B(x,1)$. 
Assume that $\phi(\St_G(x)) = F$.
Write $K := \ker(\phi)=\Fix_G(B(x,1))$ and  consider the map
$$\phi_2: K \arrow \prod^d F_1$$ given by the action of $K$ on $B(x,2)$.

Then $\phi_2(K) =\prod^a F_1$ with $a \in \{0,1,d\}$ and we have the following dichotomy:
\begin{itemize}
\item $a \in \{0,1\}$ if and only if $G$ is discrete. 
\item $a = d$ if and only if $\overline{G} = \mathrm{U}(F)$. 
\end{itemize}
\end{theorem}

We omit the proof of this theorem,
but we quote it to emphasize that in some cases
there are {\it{local conditions}} on a group $G$ which already imply the stabilization of its $k$-closures. 
Indeed, if the action of $G$ on every 1-ball is contained in $F$, then the case $a=d$ in the above theorem
yields that $\overline{G} = G^{(1)}$. 

Let $k$ be a local field  with integer ring $\mathcal{O}$, maximal ideal $\mathcal{P} \lhd \mathcal{O}$ and residue field $F = \mathcal{O}/\mathcal{P}$. 
For the action of the group $G= \PGL_2(k)$ on its Bruhat--Tits tree,
the local action is given by the group $F=GL_2(F)$ and its action on the projective line $\mathbb{P}^1F$. 
Below we discuss the fact that in this case the sequence of $k$-closures $\{G^{(k)} \ | \ k \in \mathbb{N} \}$
never stabilizes and hence gives rise to an infinite sequence of simple groups containing $G$. 
The theorem of Burger--Mozes above implies in particular that this kind of behaviour would never be possible when the ``local group'' $F$ is $\Sym(6)$, for example.

\subsection{Finding infinitely many non-discrete simple groups}
Returning to the main goal of the section, we have the following recipe to find simple subgroups of $\Aut(T)$:
\begin{enumerate}
 \item start off with some geometrically dense $G\leq \mathrm{Aut}(T)$,
 \item form its $k$-closures (which all satisfy Property $P_k$),
 \item use Theorem \ref{Tits-k} to obtain the simple subgroups $(G^{(k)})^{+_k}$. 
\end{enumerate}
We still need to ensure that these subgroups are non-discrete and different from each other, which will follow from the results below.

\begin{lemma}
 If $G\leq \mathrm{Aut}(T)$ does not stabilize a proper subtree of $T$ we have
 \begin{itemize}
  \item[(i)] $(G^{(k)})^{+_k}$ is an open subgroup of $G^{(k)}$.
  \item[(ii)] $(G^{(k)})^{+_k}$ is non-discrete if and only if $G^{(k)}$ is non-discrete.
  \item[(iii)] $(G^{(k)})^{+_k}$ satisfies Property $P_k$.
 \end{itemize}
\end{lemma}
\begin{proof}
(i) The group $(G^{(k)})^{+_k}$ is generated by fixators of $k$-edges (in particular, fixators of finite sets of vertices), which are basic open sets. 
  Hence it is open.\\
(ii) This follows from the facts that all subgroups of discrete groups are discrete and that all open subgroups of a non-discrete group are non-discrete.\\
(iii) Let $C$ be some path in $T$. Since $G^{(k)}$ satisfies $P_k$ we have that
  $$\prod_{x\in C} \Fix_{(G^{(k)})^{+_k}}(C^{k-1})_x \leq \prod_{x\in C} \Fix_{(G^{(k)})}(C^{k-1})_x=\Fix_{(G^{(k)})}(C^{k-1})$$
  and $$\Fix_{(G^{(k)})}(C^{k-1})=\bigcap (\Fix_{(G^{(k)})}(e^{k-1}) \mid e \text{ is an edge contained in } C)\leq (G^{(k)})^{+_k}.$$
\end{proof}

\begin{theorem}
 Suppose that $G\leq \mathrm{Aut}(T)$ is geometrically dense. 
 Then $(G^{(r)})^{+_r}\leq (G^{(k)})^{+_k}$ for every $r\geq k$, with equality if and only if $G^{(r)}=G^{(k)}$.
\end{theorem}
\begin{proof}
 The first claim follows from the fact that $(G^{(r)})\leq (G^{(k)})$ for every $r\geq k$. 
 
 Suppose that $G^{(r)}=G^{(k)}$ and let $g\in (G^{(k)})^{+_k}$ be a generator (so that it fixes some edge $e=(v,w)$). 
 Now, $(G^{(k)})$ satisfies $P_k$ and therefore $g=(g_1,g_2)$ where $g_1\in (G^{(k)})$ fixes $T_w$ pointwise and $g_2\in (G^{(k)})$ fixes $T_v$ pointwise. 
 In particular, there exist edges $e_1\in E(T_w)$ and $e_2\in E(T_v)$ such that $g_1$ fixes $e_1^{r-1}$ and $g_2$ fixes $e_2^{r-1}$. 
 Hence $g_1, g_2$ are generators of $(G^{(r)})^{+_r}$ and  $g\in (G^{(r)})^{+_r}$, as required.
 
 Conversely, if $(G^{(r)})^{+_r}\leq (G^{(k)})^{+_k}$ then for any $x\in G^{(k)}$ and any vertex $v$ there is some $g\in G$ such that $xg^{-1}$ fixes $B(v,k)$. 
 In particular, $xg^{-1}$ fixes $e^{k-1}$ where $e=(v,u)$ is some edge starting at $v$. 
 Thus $xg^{-1} \in (G^{(k)})^{+_k}=(G^{(r)})^{+_r}$, that is $xg^{-1}\in G^{(r)}$. 
 Since $g\in G\leq G^{(r)}$ we conclude that $x\in (G^{(r)})$.
\end{proof}

Thus, in order to construct infinitely many distinct t.d.l.c.\ simple non-discrete subgroups of Aut$(T)$ 
it suffices to find examples with infinitely many distinct $k$-closures.
By Corollary \ref{stable_k_closures}, this amounts to finding examples which do not satisfy Property $P_k$ for any $k$.

\begin{example} The Baumslag--Solitar group 
 $\mathrm{BS}(m,n):=\langle a,t \mid t^{-1}a^mt=a^n\rangle$ 
 does not satisfy $P_k$ for any $k$ when $m, n$ are coprime and
 the group acts on its Bass--Serre tree (which is isomorphic to $T_{m+n}$).
 
 Recall that $T$, the Bass--Serre tree of $\mathrm{BS}(m,n)$, has vertices the left (say) cosets of $\langle a \rangle$ and 
 (directed) edges labelled by $t^{\pm 1}$ from $u\langle a\rangle$ to $v\langle a \rangle$ if and only if there is some $i$ such that
 $v\langle a\rangle = ua^it^{\pm 1}\langle a \rangle$.
$BS(m,n)$ acts on $T$ by left multiplication and, in the non-solvable case (when neither $m$ nor $n$ equals 1),
this action is geometrically dense.

We claim that when $m,n$ are coprime $G:=BS(m,n)$ does not satisfy $P_k$ for any $k$. 
To see this (Figure \ref{A:Bass-Serre} may be helpful), consider the edge $e=(\langle a\rangle, t^{-1}\langle a \rangle)$. 
We will show that $\Fix_G(T_{\langle a \rangle)})=\Fix_G(T_{t^{-1}\langle a \rangle)})$ while $\Fix_G(e^{k-1})\neq 1$
(it contains, for instance, $a^n$).
Let $x\in \Fix_G(T_{t^{-1}\langle a \rangle)})$, then $x$ must also fix $\langle a \rangle$ as all other neighbours of $t^{-1}\langle a \rangle$ are fixed by $x$. 
Thus $x$ must be of the form $a^{*}$ where $*$ is a multiple of $n$; say $x=a^{cn^j}$ for some $c,j\in \mathbb{N}$ with $j>0$ and $c$ not divisible by $n$. 
Note that $T_{t^{-1}\langle a \rangle}$ contains vertices of the form $(t^{-1}a)^i t^{-1}\langle a \rangle$ for all $i\in\mathbb{N}$. 
Pick some $i\geq j$. 
If $x=a^{cn^j}$ fixes $(t^{-1}a)^it^{-1}\langle a \rangle$ then, 
since $ta^{cn^j}t^{-1}=(ta^nt^{-1})^{cn^{j-1}}=a^{mcn^{j-1}}$, we have
$a^{cn^j}(t^{-1}a)^it^{-1}=(t^{-1}a)^j a^{cm^j} (t^{-1}a)^{i-j} t^{-1}=(t^{-1}a)^it^{-1}a^{*}$.
For the last equality to hold, $cm^j$ must be a multiple of $n$, 
which it cannot be by the assumption that $m,n$ are coprime and the choice of $c$, unless $c=0$.
Thus we must have $x=1$. 
A similar argument yields that $\Fix_G(T_{\langle a \rangle)})=1$. 
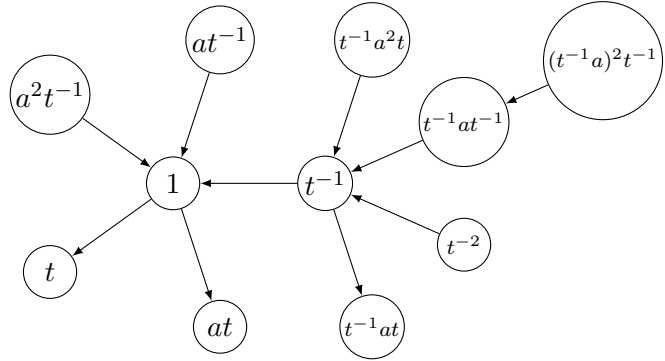
\begin{figure}\label{A:Bass-Serre}
 \begin{tikzpicture}[grow cyclic,shape=circle,level distance=20mm]
\tikzstyle{level 1}=[sibling angle=72]
\tikzstyle{level 2}=[sibling angle=48]
\tikzstyle{level 3}=[sibling angle=36]
\tikzstyle{every node}=[circle, draw,minimum size=0.7cm, inner sep=1pt]

\node[circle, draw] (root) {$1$}
	child {
		node[circle,draw] (t) {$t$} edge from parent[-latex]
		}
	child {
		node[circle,draw] (at) {$at$} edge from parent[-latex]
		}
	child {
		node[circle,draw] (t-1) {$t^{-1}$} edge from parent[latex-]
		child { node[circle, draw] (t-1at) {$\scriptstyle{t^{-1}at}$} edge from parent[-latex] }
		child { node[circle, draw] (t-2) {$\scriptstyle{t^{-2}}$} }
		child { node[circle, draw] (t-1at-1) {$\scriptstyle{t^{-1}at^{-1}}$}  
			child { node[circle, draw] (t-1at-1at-1) {$\scriptstyle{(t^{-1}a)^2t^{-1}}$} }
			}
		child { node[circle, draw] (t-1a2t) {$\scriptstyle{t^{-1}a^2t}$} edge from parent[latex-] }
		}
	child {
		node[circle,draw] (at-1) {$at^{-1}$} 
	edge from parent[latex-]
	}
	child {
		node[circle,draw] (a2t-1) {$a^2t^{-1}$} 
	edge from parent[latex-]
	};

\end{tikzpicture}

\caption{Part of the Bass-Serre tree for $BS(2,3)$. The nodes are labelled by their coset representatives.
The arrows on the edges indicate travelling in the $t$ direction.}
\end{figure}
\end{example}

\begin{example}
The group $G = \mathrm{PSL}(2,\mathbb{Q}_p)$ acting on its Bruhat--Tits tree (which is isomorphic to $T_{p+1}$) also does not satisfy $P_k$ for any $k$.
Indeed it is well known that the action of $G$ on $\partial T$ is isomorphic to the action of the same group on the projective line $\mathbb{P}^{1}(\mathbb{Q}_p)$ by fractional linear transformations. 
In particular the stabilizer of three boundary points is trivial.
This means that an element of $G$ is completely determined by its action on three distinct points of $\partial T$ 
(these elements should be thought of as $p$-adic M\"{o}bius transformations). 

Now let $C = (\ldots,c_{-1},c_0,c_1,c_2, \ldots)$ be an infinite or finite geodesic. 
And assume that $(\ldots f_0,f_1,f_2, \ldots) \in \prod_{i=-\infty}^{\infty} F_i$ with $f_i \in \Fix_G(C^{k-1})|_{T_i}$. 
Then each such $f_i$ is defined on $\partial T_i$ which contains many boundary points.
Hence each $f_i$ admits a unique extension to the whole tree.
This is a strong obstruction to satisfying Property $P_k$, for any $k$. 

\end{example}

We note that this method finds infinitely many t.d.l.c.\ simple non-discrete groups which are pairwise distinct as subgroups of Aut$(T)$. 
It would be desirable to know whether these subgroups are pairwise non-isomorphic. This is stated as work in progress in \cite{BEW11}. 
Using different methods, S.\ Smith (\cite{S14}) has found uncountably many t.d.l.c.\ simple non-discrete groups which are pairwise non-isomorphic.
This is discussed by C.\ Reid and G.\ Willis in their chapter. 

\end{document}